\documentclass[a4paper,11pt]{amsart}
\usepackage[inner=2.3cm,outer=2.3cm, bottom=3.3cm]{geometry}
\usepackage{amsmath,amsgen,amstext,amsbsy,amsopn,amsfonts, enumitem, amssymb, amsthm}  
\usepackage{url,graphicx,tabularx,array, tikz-cd, setspace}
\definecolor{leaf}{rgb}{0,.35,0}
\definecolor{chianti}{rgb}{0.6,0,0}
\definecolor{meretale}{rgb}{0,0,.6}
\usepackage[colorlinks=true,linkcolor=meretale,
urlcolor=chianti, citecolor=leaf,pagebackref]{hyperref}

\numberwithin{equation}{subsection}

\newcommand{\RHom}{\operatorname{\mathbf{R}Hom}}
\newcommand{\Hom}{\operatorname{Hom}}
\newcommand{\ldt}{\otimes^\mathbf{L}}
\newcommand{\depth}{\operatorname{depth}}
\newcommand{\sD}{\mathsf{D}}
\newcommand{\rb}{\mathsf{b}}
\newcommand{\rf}{\mathsf{f}}
\newcommand{\Dbf}{\sD_\rf^\rb}
\newcommand{\mx}{\mathfrak{m}}
\newcommand{\px}{\mathfrak{p}}
\newcommand{\Spec}{\operatorname{Spec}}
\newcommand{\injdim}{\operatorname{id}}
\newcommand{\projdim}{\operatorname{pd}}
\newcommand{\Ext}{\operatorname{Ext}}
\newcommand{\Tor}{\operatorname{Tor}}
\newcommand{\amp}{\operatorname{amp}}
\newcommand{\HH}{\textup{H}}
\newcommand{\HR}{\HH^0(R)}
\newcommand{\HS}{\HH^0(S)}
\newcommand{\HD}{\HH^0(D)}
\newcommand{\HE}{\HH^0(E)}
\newcommand{\Hc}[1][\mx]{\operatorname{H}_{#1}}

\newcommand{\dTorsion}[1][\mx]{\operatorname{\mathbf{R}\Gamma}_{#1}}
\newcommand{\dComplete}[1][\mx]{\operatorname{\mathbf{L}\Lambda}^{#1}}

\newcommand{\tel}{\operatorname{Tel}}
\newcommand{\Supp}{\operatorname{Supp}}
\newcommand{\Rhat}{\widehat{R}}

\newtheorem{defn}{Definition}[section]
\newtheorem{theorem}{Theorem}[section]
\newtheorem*{varthmA}{Theorem A}
\newtheorem*{varthmB}{Theorem B}

\newtheorem{lemma}[theorem]{Lemma}

\title{Maximal Cohen-Macaulay DG-Complexes}
\author{Zachary Nason}
\address{Department of Mathematics, University of Nebraska, Lincoln, NE 68588-0130, USA}
\email{znason2@huskers.unl.edu}
\thanks{Part of this article was written while the author was supported by the ``RTG: Commutative Algebra at Nebraska'' grant (NSF grant 2342256).}

\begin{document}	
\begin{abstract}
Let $R$ be a commutative noetherian local differential graded (DG) ring. In this paper we propose a definition of a maximal Cohen-Macaulay DG-complex over $R$ that naturally generalizes a maximal Cohen-Macaulay complex over a noetherian local ring, as studied by Iyengar, Ma, Schwede, and Walker. Our proposed definition extends the work of Shaul on Cohen-Macaulay DG-rings and DG-modules, as any maximal Cohen-Macaulay DG-module is a maximal Cohen-Macaulay DG-complex. After proving necessary lemmas in derived commutative algebra, we establish the existence of a maximal Cohen-Macaulay DG-complex for every DG-ring with constant amplitude that admits a dualizing DG-module. We then use the existence of these DG-complexes to establish a derived Improved New Intersection Theorem for all DG-rings with constant amplitude.
\end{abstract}

\maketitle

\section{Introduction}

In this article, all rings are unital and commutative. The fundamental objects of study in commutative algebra are rings, and modules defined over rings. In derived (or higher) commutative algebra, a ring is generalized to a differential graded ring (DG-ring), which is a differential graded algebra concentrated in non-positive cohomological degrees. The modules of commutative algebra are correspondingly replaced with DG-modules over DG-rings. DG-rings have been studied for several decades, with papers by Avramov and Foxby \cite{LocalG}, and F\'{e}lix, Halperin, and Thomas \cite{GorSpace} using DG-rings to solve problems in commutative algebra and topology, respectively. In the early 2000s, work by authors such as Frankild, Iyengar, and J\o rgensen \cite{DDGA,GDGA} proved results involving DG-rings, particularly by developing a suitable definition of a Gorenstein DG-ring and by proving that the Koszul complex of a sequence of elements in a Gorenstein local ring is a Gorenstein DG-ring. Working from a different perspective, T\"{o}en and Vezzosi \cite{HAG} incorporated DG-rings in their construction of homotopical algebraic geometry. However, Yekutieli was the first to develop a formal theory of derived commutative algebra as a subject in its own right. As a result of Yekutieli's initial work, the study of derived commutative algebra has seen a rise in interest over the past few years. In particular, Shaul has recently developed a theory of generalized injective DG-modules which naturally expands the Bass structure theorem in commutative algebra, and has developed a suitable definition of a Cohen-Macaulay DG-ring and a Cohen-Macaulay DG-module.

In \cite{CMDGA}, the definition of a Cohen-Macaulay DG-module has an important limitation as noted in section 6. In order for $M \in \Dbf(R)$ to be a Cohen-Macaulay DG-module, it must have the same cohomological amplitude as its underlying DG-ring $R$. This limitation makes Shaul's definition of a Cohen-Macaulay DG-module a natural generalization of a Cohen-Macaulay module over a noetherian local ring. However, as Shaul asks in \cite{CMDGA}, is it possible to define a Cohen-Macaulay DG-module in a way that generalizes a Cohen-Macaulay \emph{complex} over a noetherian local ring? (For reference, a Cohen-Macaulay complex $M$ over a noetherian local ring is a complex such that the cohomological amplitude of $\dTorsion(M)$ is zero.)

To help answer this question, we have taken inspiration from a 2018 paper by Iyengar, Ma, Schwede, and Walker \cite{MCMC}. In this paper, the authors defined a maximal Cohen-Macaulay complex $M$ over a noetherian local ring $R$ in terms of the local cohomology of $M$ and the existence of a nonzero map between $\HH^0(M) \to \HH^0(k \ldt_R M)$. The authors then prove that such complexes exist over rings with dualizing complexes, and that their existence implies the Improved New Intersection Theorem. Since the authors used derived category techniques to prove their theorems, it is possible to generalize many definitions and techniques from commutative algebra to derived commutative algebra. In particular, we are able to formulate a definition of a maximal Cohen-Macaulay DG-complex that allows the cohomological amplitude of the DG-complex to exceed the amplitude of the DG-ring. From this definition, we then show that their existence implies a derived version of the Improved New Intersection Theorem for a large class of DG-rings (DG-rings whose cohomological amplitude remains invariant under localization - referred to in the literature as having constant amplitude.)

The natural next step is to then show that maximal Cohen-Macaulay DG-complexes exist over noetherian local DG-rings. By using a version of the Canonical Element Theorem proved in \cite{MCMC}, we are able to show that maximal Cohen-Macaulay DG-complex exist for every DG-ring with constant amplitude that admits a dualizing DG-module.

\begin{varthmA}
Let $R$ be a noetherian local DG-ring with constant amplitude and $\amp(R) = n$. If $D$ is a dualizing DG-module with $\inf(D) = 0$, then the DG-module $\RHom_R(\Sigma^n(D^{\leq n}), D)$ is a maximal Cohen-Macaulay DG-complex.
\end{varthmA}

Using the existence of maximal Cohen-Macaulay DG-complexes, we are able to conclude that the derived Improved New Intersection Theorem holds for all DG-rings with constant amplitude.

\begin{varthmB}[Derived Improved New Intersection Theorem]
Let $R$ be a noetherian local DG-ring with $\amp(R) = n$ with constant amplitude (i.e., $\Supp(\HH^{\inf(R)}(R)) = \Spec(\HR)$), and $F \in \Dbf(R)$ a bounded semi-free DG-module such that $\HH^0(F) \neq 0$ and $\HH^i(F)$ is of finite length for $i \leq -1$. If an ideal $I$ in $\HR$ annihilates a minimal generator of $\HH^0(F)$, then $\projdim(F)+n \geq \dim \HR - \dim(\HR/I)$. 
\end{varthmB}

As demonstrated in this article, the proofs of both Theorem A and B crucially use the assumption of constant amplitude, and we do not see a way at present to lift this assumption.

\section{Background}

\subsection{DG-Rings}

In this section, we will review a number of fundamental definitions and results in derived commutative algebra. More detailed references can be found in Shaul's work \cite{CMDGA,DerCom, Injective} and in Yekutieli's book \cite{DC}. In addition, Christensen, Foxby, and Holm's textbook \cite{DCMCA} is an excellent and thorough source on derived category methods in (standard) commutative algebra, which often can be applied in derived commutative algebra.

Throughout this article, we will denote all grading cohomologically in order to fit in with the majority of current literature. Occasionally, this will conflict with convention regarding certain algebraic objects (like the Koszul complex), but maintaining a uniform cohomological grading will hopefully prevent any ambiguity.

A differential graded algebra (DGA) is a graded algebra
\begin{equation*}
R = \bigoplus_{i \in \mathbb{Z}} R^i
\end{equation*}
equipped with a differential $\delta^i: R^i \to R^{i+1}$ such that $\delta^{i+1}\circ\delta^i = 0$ and such that the differential satisfies the Leibniz rule
\begin{equation*}
\delta(rs) = \delta(r)s + (-1)^{|r|}r\delta(s)
\end{equation*}
If a differential graded algebra satisfies a graded version of commutativity ($rs = (-1)^{|r||s|}sr$), then the DGA is graded-commutative. If in addition $r^2 = 0$ for all elements of odd degree, then the DGA is strictly graded-commutative. Throughout the rest of the article, all DGAs will be strictly graded-commutative and will be referred to as ``commutative'' for simplicity.

The main object of study in this article is a noetherian local differential graded ring (DG-ring). This is a commutative differential graded algebra that generalizes a noetherian local ring.

\begin{defn}
Let $R$ be a commutative differential graded algebra. Then $R$ is a \emph{noetherian local differential graded ring} if the following hold:
\begin{enumerate}
\item $R$ is non-positive ($R^i = 0$ for all $i > 0$)
\item $R$ is cohomologically bounded
\item $\HR$ is a noetherian local ring with maximal ideal $\mx$ and residue field $k$
\item $\HH^i(R)$ is a finitely generated $\HR$-module for all $i \in \mathbb{Z}$ 
\end{enumerate}
\end{defn}

Throughout the rest of the paper, we will simply refer to these objects as ``DG-rings'' for simplicity.

For a DG-ring $R$, a differential graded module (DG-module) is a graded module equipped with a differential compatible with the differential of $R$. The collection of all DG-modules over $R$ and all chain maps between these DG-modules forms a category $\mathsf{DG}(R)$. This category is analogous to the category of chain complexes over a ring. By formally inverting the quasi-isomorphisms in $\mathsf{DG}(R)$, we obtain the derived category of DG-modules $\sD(R)$. This is a triangulated category with respect to suspension $\Sigma$. The objects of $\sD(R)$ with bounded cohomology form a full subcategory of $\sD(R)$, which we will denote as $\sD^\rb(R)$. Similarly, the objects of $\sD(R)$ with finitely generated cohomology form a full subcategory of $\sD(R)$, which we will denote as $\sD_\rf(R)$. The intersection of these two subcategories, denoted $\Dbf(R)$, is itself a full subcategory of $\sD(R)$, and consists of all DG-modules with bounded and finitely generated cohomology. For a DG-module $M \in \Dbf(R)$ (or more generally in $\sD^\rb(R)$), the largest and smallest nonzero cohomologies are finite, and are denoted $\sup(M)$ and $\inf(M)$ respectively. The quantity $\sup(M)-\inf(M)$ is also finite, and is called the amplitude of $M$ ($\amp(M)$).

For every DG-module $M \in \sD(R)$, the smart truncations $M^{\leq n}$ and $M^{>n}$ are themselves objects in $\sD(R)$ such that 

\begin{equation*}
\HH^i(M^{\leq n}) = \begin{cases}&\HH^i(M) \text{ if } i \leq n \\ &0 \text{ otherwise}
\end{cases}
\end{equation*}
and similarly
\begin{equation*}
	\HH^i(M^{> n}) = \begin{cases}&\HH^i(M) \text{ if } i > n \\ &0 \text{ otherwise}
	\end{cases}
\end{equation*}

As first described in \cite{Tilting}, it is possible to localize a DG-ring $R$ with respect to a prime ideal $\px$ in $\Spec(\HR)$. The localization $R_\px$ is itself a DG-ring with $\HH^i(R_\px) = \HH^i(R)_\px$ for all $i \in \mathbb{Z}$. This implies that $\amp(R_\px) \leq \amp(R)$ with equality holding if $\px \in \Supp(\HH^{\inf(R)}(R))$. Difficulties often arise when there exists a prime $\px \in \Spec(\HR)$ such that $\amp(R_\px) < \amp(R)$, which will occasionally require us to restrict our attention to DG-rings where this is impossible (i.e., $\Supp(\HH^{\inf(R)}(R)) = \Spec(\HR)$). Such DG-rings are said to have constant amplitude, and were first defined by Shaul in \cite{Koszul}. We will work with DG-rings of constant amplitude later in the paper. 

Localization with respect to $\px$ can be extended to a DG-module $M \in \sD(R)$ by letting $M_\px = M \otimes_R R_\px \in \sD(R_\px)$. As with DG-rings, we have $\HH^i(M_\px) = \HH^i(M)_\px$ for all $i \in \mathbb{Z}$, which in turn implies that $\amp(M_\px) \leq \amp(M)$.
\subsection{Derived Hom and Derived Tensor}

Let $R$ be a DG-ring and $M$ a DG $R$-module. The classical functors $\Hom_R(M, -)$, $\Hom_R(-, M)$, and $- \otimes_R M$  (which operate on objects and morphisms in $\mathsf{DG}(R)$) induce derived functors $\RHom_R(M, -)$, $\RHom_R(-, M)$, and $- \ldt_R M$ which operator on objects and morphisms in $\sD(R)$. These derived functors are computed using semi-projective, semi-injective, and semi-flat replacements, which exist for any DG $R$-module. Explicitly,
\begin{align*}
	\RHom_R(M, N) &= \Hom_R(P, N) \text{ (where $P$ is a semi-projective replacement of $M$)} \\
	\RHom_R(N, M) &= \Hom_R(N, I) \text{ (where $I$ is a semi-injective replacement of $M$)} \\
	M \ldt_R N &= F \otimes_R N \text{ (where $F$ is a semi-flat replacements of $M$)}
\end{align*}
The cohomologies of the derived Hom functor are denoted with $\Ext$, and the cohomologies of the derived tensor functor are denoted by $\Tor$:
\begin{align*}
	\Ext^i_R(M, N) &= \HH^i(\RHom_R(M, N)) \\
	\Tor_i^R(M, N) &= \HH^{-i}(M \ldt_R N)
\end{align*}
When $R$ is a local ring, the above definitions of $\Ext$ and $\Tor$ exactly coincide with the standard definitions of $\Ext$ and $\Tor$ in commutative algebra.

In derived commutative algebra, the projective and injective dimension of a DG-module are defined using Ext and Tor.
\begin{defn}
	Let $R$ be a DG-ring, and $M \in \sD(R)$. The \emph{projective dimension} of $M$ is the number
	\begin{equation*}
		\inf\{n \in \mathbb{Z} \, | \, \Ext^i_R(M, N) = 0 \text{ for any $N \in \sD^\rb(R)$ and any $i > n + \sup{N}$}\}
	\end{equation*}
	and the \emph{injective dimension} of $M$ is the number
	\begin{equation*}
		\inf\{n \in \mathbb{Z} \, | \, \Ext^i_R(N, M) = 0 \text{ for any $N \in \sD^\rb(R)$ and any $i > n - \inf{N}$}\}
	\end{equation*}
\end{defn}
\subsection{Derived Torsion and Derived Completion}

Let $R$ be a DG-ring, and let $I$ be an ideal in $\HR$. In contrast to the derived Hom and derived tensor functors, defining and computing the derived $I$-torsion and derived $I$-completion functors over DG-rings is more involved. The following paragraph gives a brief sketch of how these derived functors are constructed. A formal construction is given in \cite{DerCom}, particularly in section 2.

The collection of all DG-modules whose cohomologies are $I$-torsion forms a triangulated subcategory of $\sD(R)$, denoted as $\sD_{I-\text{tor}}(R)$. The inclusion functor $J: \sD_{I-\text{tor}}(R) \to \sD(R)$ has a right adjoint $R: \sD(R) \to \sD_{I-\text{tor}}(R)$. The derived $I$-torsion functor $\dTorsion[I]$ is then defined to be the composition $J \circ R$. The derived functor $\dTorsion[I]$ itself has a left adjoint which is the derived $I$-completion, denoted $\dComplete[I]$.

Using the properties of weakly proregular sequences and the derived base change property of the telescope complex, one can explicitly  compute the derived $I$-torsion and derived $I$-completion of a DG-module over any DG-ring by using telescope complexes, as shown by Shaul in \cite{DerCom}. In particular, we have that
\begin{align*}
	\dTorsion[I](M)  &\simeq \tel(R^0; J) \otimes_{R^0} M \\
	\dComplete[I](M) &\simeq \Hom_{R^0}(\tel(R;J), M)
\end{align*}
where $J$ is a finitely generated ideal in $R^0$ whose image with respect to the map $R^0 \to \HR$ generates $I$.
The cohomologies of the derived $I$-torsion functor are denoted by $\HH_I^i$, and the cohomologies of the derived $I$-complete functor are denoted by $\HH_i^I$:
\begin{align*}
	\HH_I^i(M) &= \HH^i(\dTorsion[I](M)) \\
	\HH^I_i(M) &= \HH^{-i}(\dComplete[I](M))
\end{align*}
When $R$ is a local ring, the above definitions of $\HH_I^i$ and $\HH^I_i$ exactly coincide with the standard definitions of local cohomology and local homology.

\subsection{Dualizing DG-Modules}

\begin{defn}
\label{Dualizing}
Let $R$ be a DG-ring. A DG-module $D$ is a \emph{dualizing DG-module} if the following three conditions hold:
\begin{enumerate}
	\item $D$ is finite (i.e., $D \in \Dbf(R)$).
	\item The homothety morphism $R \to \RHom_R(D, D)$ is an isomorphism.
	\item $\injdim(D) < \infty$
\end{enumerate}
\end{defn}
If $R$ is a local ring, then a dualizing DG-module is just a dualizing complex over $R$. 

Throughout the rest of the article, we will often construct the dual of a DG-module with respect to a dualizing DG-module. To be concise, we will set
\begin{equation*}
	M^\dagger := \RHom_R(M, D)
\end{equation*}
where $M$ is an arbitrary DG $R$-module, and $D$ is a dualizing DG $R$-module.

For a dualizing DG-module $D$ over a DG-ring $R$, the biduality morphism
\begin{equation*}
	M \to \RHom_R(\RHom_R(M, D), D)
\end{equation*}
is an isomorphism for all $M \in \Dbf(R)$. This fact follows from the definition of a dualizing DG-module and derived Hom evaluation.

If the DG-ring $R$ itself satisfies Definition~\ref{Dualizing}, then $R$ is called a Gorenstein DG-ring. Since the homothety morphism $R \to \RHom_R(R, R)$ is always an isomorphism in this case, the only non-trivial condition is that $R$ has finite injective dimension, which is why the label ``Gorenstein'' makes sense.

It is a well-known fact in commutative algebra that a noetherian local ring has a dualizing complex if and only if it is the homomorphic image of a Gorenstein local ring. In \cite{RecognizingDC}, J\o rgensen proves that Gorenstein DG-rings can take the place of Gorenstein local rings (i.e., a noetherian local ring has a dualizing complex if and only if it is the homomorphic image of a Gorenstein DG-ring). Using J\o rgensen's proof, we can extend this result to arbitrary DG-rings and show that a DG-ring has a dualizing DG-module if and only if it is the homomorphic image of a Gorenstein DG-ring.

If two DG-modules are dualizing DG-modules, then they are isomorphic up to a shift; conversely, any shift of a dualizing DG-module is itself a dualizing DG-module. Because shifting a dualizing DG-module preserves the dualizing property, it is occasionally important to specify the appropriate position of a dualizing DG-module. To this end, authors typically use the phrase ``normalized dualizing DG-module'' (or normalized dualizing complex) to refer to a specific shift. Unfortunately, there are two common shifts that ``normalized'' can mean in the literature. The first shift (used by Hartshorne) is the shift where $\inf(D) = -\dim(\HR)$. The second shift (used by \cite{MCMC}) is the shift where $\inf(D) = 0$. We will use both shifts throughout the rest of the article, and so to avoid confusion, we will call the shift where $\inf(D) = -\dim(\HR)$ ``left-normalized'' and the shift where $\inf(D) = 0$ ``right-normalized''.
\section{The Improved New Intersection Theorem for DG-Rings}

In \cite{MCMC}, Iyengar, Ma, Schwede, and Walker show that the existence of complexes of maximal depth over noetherian local rings implies the Improved New Intersection Theorem. In this section, we will define a DG-module of maximal depth, and will show that the existence of such DG-modules implies a derived analogue of the Improved New Intersection Theorem. We will list a few important definitions and lemmas that we will need for the main proof.

\subsection{Depth and Projective Dimension}

\begin{defn}[Depth and Sequential Depth]
Let $R$ be a noetherian local DG-ring with an ideal $I \in \HR$ , and let $M \in \sD^+(R)$. The I-\emph{depth} of $M$ is 
\begin{equation*}
\depth(I, M) = \inf(\RHom_R(\HR/I, M)) = \inf(\dTorsion[I](M))
\end{equation*}
The \emph{depth} (or $\mx$-depth) of $M$ is
\begin{equation*} \depth(M) = \inf(\RHom_R(k, M))=\inf(\dTorsion(M)).
\end{equation*} 
\end{defn}

The following two proofs are analogues of Remark 2.3 in \cite{MCMC}

\begin{lemma}
\label{DepthBound}
Let $R$ be a DG-ring with an ideal $I \subset \HR$, and let $M \in \sD^+(R)$. Then $\depth(I, M) \geq \inf(M)$, and equality holds if $\Gamma_I(\HH^{\inf(M)}(M)) \neq 0$.
\end{lemma}

\begin{proof}
(Note that this is a generalization of \cite[Prop. 3.2]{CMDGA}.) By derived Hom-tensor adjunction, we have
\begin{align*}
\RHom_R(\HR/I, M)   &\simeq \RHom_R(\HR/I \ldt_{\HR} \HR, M) \\
                    &\simeq \RHom_{\HR}(\HR/I, \RHom_R(\HR, M))
\end{align*}
Thus, $\depth_R(I, M) = \depth_{\HR}(I, \RHom_R(\HR, M))$. Since
\begin{equation*}\depth_{\HR}(I, \RHom_R(\HR, M)) \geq \inf(\RHom_R(\HR, M))
\end{equation*}
by \cite[Theorem 14.3.16]{DCMCA}, we have 
\begin{equation*}
\depth_R(I, M) \geq \inf(\RHom_R(\HR, M))
\end{equation*}
By \cite[Prop. 3.3]{Injective}, $\inf(M) = \inf(\RHom_R(\HR, M))$, and so $\depth_R(I, M) \geq \inf(M)$.

Now suppose $\Gamma_I(\HH^{\inf(M)}(M)) \neq 0$. By \cite[Prop. 3.3]{Tilting}, we have 
\begin{equation*}
\HH^{\inf(M)}(M) \cong \HH^{\inf(M)}(\RHom_R(\HR, M))
\end{equation*}
and so $\Gamma_I(\HH^{\inf(M)}(\RHom_R(\HR, M)) \neq 0$. By applying \cite[Theorem 14.3.16]{DCMCA} to the $\HR$-complex $\RHom_R(\HR, M)$, we have that 
\begin{equation*}
\depth_{\HR}(I, \RHom_R(\HR, M)) = \inf(M)
\end{equation*}
and so $\depth_R(I, M) = \inf(M)$ by the previous paragraph.
\end{proof}
\begin{theorem}[The Derived Auslander-Buchsbaum Formula]
\label{ABF}
Let $R$ be a DG-ring, and let $M, F \in \Dbf(R)$ with $F$ semi-free and with $\sup(F) = 0$. Then 
\begin{equation*}
\depth(M \ldt_R F) = \depth(M) - \projdim(F)
\end{equation*}
\end{theorem}

\begin{proof}
(This proof is essentially the same as \cite[Thm. 2.4]{Amp}. In addition, \cite[Thm 2.16]{Min} proves essentially the same result using a slightly different definition of depth.)

Consider the DG-modules $\RHom_R(k, M)$ and $k \ldt_R F$. These can be regarded as $k$-complexes, so we have the following isomorphism of $k$-complexes
\begin{equation*}
\RHom_R(k, M) \ldt_k (k \ldt_R F) \simeq \RHom_R(k, M \ldt_R F)
\end{equation*}
Thus, we have  
\begin{equation*}
\inf(\RHom_R(k, M)) + \inf(k \ldt_R F) = \inf(\RHom_R(k, M \ldt_R F))
\end{equation*}
Since $\projdim_R(F) < \infty$, we have that
\begin{equation*}
\RHom_R(k, M) \ldt_R F \simeq \RHom_R(k, M \ldt_R F)
\end{equation*}
by \cite[Prop. 2.2]{HomDim}. This implies that 
\begin{align*}
\inf(\RHom_R(k, M) \ldt_R F) 	&= \inf(\RHom_R(k, M \ldt_R F)) \\
								&= \inf(\RHom_R(k, M)) + \inf(k \ldt_R F)
\end{align*}
By the definition of depth, this gives $\depth(M \ldt_R F) = \depth(M) + \inf(k \ldt_R F)$. By \cite[Cor. 2.3]{Min}, we have that $\inf(k \ldt_R F) = -\projdim_R(F) + \sup(F) = -\projdim_R(F)$, and so we obtain
\begin{equation*}
\depth(M \ldt_R F) = \depth(M) - \projdim(F)
\end{equation*}
\end{proof}
\subsection{Derived Completion}

An important property of derived completion is a version of Nakayama's Lemma for the cohomologies of derived complete DG-modules.

\begin{lemma}[Nakayama's Lemma for Derived Complete DG-Modules]
\label{NAK}
Let $R$ be a noetherian local DG-ring, and let $I$ be an ideal in $\HR$. Suppose that $M \in \mathsf{D}(R)$ is $I$-derived complete (i.e., $M \simeq \dComplete[I](M)$). If $I\HH^i(M) = \HH^i(M)$, then $\HH^i(M) = 0$.
\end{lemma}
\begin{proof}
Since $I$ is a finitely generated ideal in $\HR$, we can choose some finitely generated ideal $J$ in $R^0$ such that the image of $J$ in $\HR$ is $I$. Denote the list of finite generators of $J$ with $\mathbf{j}$. Since $M$ is $I$-derived complete, we have that $M \simeq \Hom_{R^0}(\tel(R^0; \mathbf{j}), M)$ in $\sD(R)$. The telescope complex $\tel(R^0; \mathbf{j})$ is a semi-free replacement (over $R^0$) of the \v{C}ech complex $\check{C}(\mathbf{j})$ by \cite[Lemma 5.7]{HCT}. Thus, the complex $\Hom_{R^0}(\tel(R^0; \mathbf{j}), M)$ in $\sD(R^0)$ is isomorphic to $\RHom_{R^0}(\check{C}(\mathbf{j}), M)$. Thus, $M \simeq \RHom_{R^0}(\check{C}(\mathbf{j}), M)$ in $\sD(R^0)$, which implies that $M$ is $J$-derived complete in $\sD(R^0)$. By Lemma 15.92.1 in \cite[Tag091N]{stacks-project}, a complex is derived complete if and only if each cohomology of the complex is derived complete. Thus, the cohomologies of $M$ are derived complete with respect to $J$. If $I\HH^i(M) \cong \HH^i(M)$, then $J\HH^i(M) \cong \HH^i(M)$ through restriction of scalars, and so $\HH^i(M) = 0$ by Lemma 15.92.7 in \cite[Tag091N]{stacks-project}.
\end{proof}

Using Nakayama's lemma for derived complete DG-modules, we can prove a fact about Koszul DG-modules over DG-rings that is already known for Koszul complexes over rings \cite[Lemma 2.2]{MCMC}\cite[Lemma 14.3.5]{DCMCA}.

\begin{lemma}
\label{KoszulInf}
Let $R$ be a DG-ring and $I$ be an ideal in $\HR$. Suppose that $M$ is a derived $I$-complete DG-module. For any sequence $\mathbf{r} = r_1, \ldots, r_n$ in $I$, we have
\begin{equation*}
\inf(K(\mathbf{r}; M)) \leq \inf(M)
\end{equation*}
\end{lemma}
\begin{proof}
We first show that $K(\mathbf{r}; M) :=K(\mathbf{r}) \ldt_R M$ is derived $I$-complete. Let $J$ be a finitely generated ideal in $R^0$ such that the image of $J$ in $\HR$ is $I$ (as in the previous lemma). Denote the finite list of generators of $J$ with $\mathbf{j}$. We have an isomorphism $M \simeq \RHom_{R^0}(\check{C}(\mathbf{j}), M)$ in $\sD(R)$ since $M$ is derived $I$-complete. The DG $R$-module $\RHom_{R^0}(\check{C}(\mathbf{j}), M)$ is isomorphic to $\RHom_R(\check{C}(\mathbf(j) \ldt_{R^0} R, M))$ through derived Hom-tensor adjunction. Thus, we have the isomorphism
\begin{equation*}
M \simeq \RHom_R(\check{C}(\mathbf{j}) \ldt_{R^0} R, M) 
\end{equation*}
Applying the functor $-\ldt_R K(\mathbf{r})$ to the above isomorphism yields
\begin{align*}
K(\mathbf{r}; M)	&\simeq \RHom_R(\check{C}(\mathbf{j}) \ldt_{R^0} R, M) \ldt_R K(\mathbf{r}) \\
					&\simeq \RHom_R(\check{C}(\mathbf{j}) \ldt_{R^0} R, M\ldt_R K(\mathbf{r})) \\
					&\simeq \RHom_{R^0}(\check{C}(\mathbf{j}), M \ldt_R K(\mathbf{r}))\\
					&\simeq \dComplete[I](K(\mathbf{r}); M) 
\end{align*}
where the second isomorphism follows from derived tensor evaluation \cite[Lemma 2.7]{FinDim} (noting that $K(\mathbf{r})$ has finite projective dimension over $R$), and the third isomorphism follows from derived Hom-tensor adjunction. Thus, $K(\mathbf{r};M)$ is derived $I$-complete.

Since the Koszul DG-module over a sequence of elements is iteratively constructed, it suffices to show the result on the Koszul DG-module over a single element $r \in I$. By Definition 2.4 in \cite{Koszul}, the Koszul DG-module $K(r;M)$ fits in the distinguished triangle
\begin{center}
\begin{tikzcd}[column sep = small]
M \arrow[r, "\cdot r"] & M \arrow[r] & K(r; M)  \arrow[r] & \phantom{}
\end{tikzcd}
\end{center}
This gives us a long exact sequence in cohomology
\begin{center}
\begin{tikzcd}[column sep = small]
\cdots \arrow[r] & \HH^i(M) \arrow[r, "\cdot r"] & \HH^i(M) \arrow[r] & \HH^i(K(r; M)) \arrow[r] & \HH^{i+1}(M) \arrow[r] & \cdots
\end{tikzcd}
\end{center}

If $\HH^i(K(r; M)) = 0$, then $\HH^i(M) \cong r\HH^i(M)$ from the long exact sequence, which forces $\HH^i(M) = 0$ by Lemma~\ref{NAK}. This proves the result.
\end{proof}

\subsection{DG-Modules of Maximal Depth}

We now define a derived analogue of complexes of maximal depth (as defined in \cite{MCMC}).
\begin{defn}[DG-Modules of Maximal Depth]
Let $R$ be a DG-ring with $\amp(R) = n$, and let $M \in \sD^{\rb}(R)$. The DG-module $M$ has maximal depth if
\begin{enumerate}
\item The map of $\HR$-modules $\HH^n(M) \to \HH^n(k \ldt_R M)$ is nonzero
\item $\depth(M) = \dim(\HR)$
\end{enumerate}
\end{defn}

\begin{lemma}
\label{split}
Let $R$ be a DG-ring, and let $M$ be a DG-module of maximal depth. For any $DG$-module $F$ with $\sup(F) = 0$, if $\HH^0(F) \otimes_R k$ is nonzero, then so is $\HH^n(F \ldt_R M)$.
\end{lemma}
\begin{proof}
The idea of the proof given in \cite[Lemma 3.1]{MCMC} applies more generally with DG-rings. Let $P \simeq M$ be a semi-projective replacement for $M$. Since $F \otimes_R P \simeq F \ldt_R M$, we have $\HH^n(F \otimes_R P) \cong \HH^n(F \ldt_R M)$, and so it suffices to show that $\HH^n(P \otimes_R k)$ is nonzero. Thus, we can assume that $M$ is semi-projective.   By hypothesis we can find a cycle $z \in M^n$ whose image in $k \otimes_R M$ is not a boundary, and then we can construct a map of DG-modules $R \to \Sigma^n M$ where $r \to rz$. We can construct (in a similar fashion as \cite[Lemma 3.1]{MCMC}) the commutative diagram
\begin{equation}\label{Diag1}
\begin{tikzcd}
R \arrow[r, "\cdot z"]\arrow[d] & \Sigma^n M \arrow[d] \\
k \arrow[r, shift left]\arrow[r, leftarrow, dashed, shift right] & k \otimes_R \Sigma^n M
\end{tikzcd}
\end{equation}
where the dashed line is a left inverse of the arrow above it. To see that the dashed line exists, consider the map $k \to k \otimes_R \Sigma^n M$. There is a nonzero splitting \begin{tikzcd} k \arrow[r, shift left]\arrow[r, leftarrow, dashed, shift right] & \HH(k \otimes_R \Sigma^n M) \end{tikzcd} in $\sD(k)$ since the image of $z$ in $k \otimes_R M$ is not a boundary. Since $\HH(k \otimes_R \Sigma^n M) \simeq k \otimes_R \Sigma^n M$ in $\sD(k)$, we have a splitting \begin{tikzcd} k \arrow[r, shift left]\arrow[r, leftarrow, dashed, shift right] & k \otimes_R \Sigma^n M \end{tikzcd} which lifts to $\sD(R)$ through restriction of scalars. By applying the $F \ldt_R -$ functor and taking the zeroth cohomology to Diagram \ref{Diag1}, we obtain the commutative diagram
\begin{equation} \label{splitDiag}
\begin{tikzcd}
\HH^0(F) \arrow[r]\arrow[d] & \HH^0(F \otimes_R \Sigma^n M) \arrow[d] \\
\HH^0(F \ldt_R k) \arrow[r, shift left] \arrow[r, leftarrow, dashed, shift right] & \HH^0(F \ldt_R (k \otimes_R \Sigma^n M))
\end{tikzcd}
\end{equation}
By hypothesis the left vertical arrow is nonzero, and so its composition with the horizontal arrow is nonzero. Thus, $\HH^0(F \ldt_R \Sigma^n M) \cong \HH^n(F \ldt_R M)$ is nonzero. 
\end{proof}    
The condition of maximal depth is invariant under derived completion.

\begin{lemma}
\label{CompleteMD}
Let $R$ be a DG-ring, and let $M$ be a DG-module of maximal depth. Then $\dComplete(M)$ has maximal depth.
\end{lemma}
\begin{proof}
Since the DG $R$-module $k$ is $\mx$-torsion, we have that $\dTorsion(k) \simeq k$ \cite[Prop. 2.9]{DerCom}. Thus, 
\begin{align*}
\RHom_R(k, M) &\simeq \RHom_R(\dTorsion(k), M)\\
                &\simeq \RHom_R(M, \dComplete(M))
\end{align*}
by Greenless-May duality \cite[Thm. 2.12]{DerCom}. Thus, $\depth(\dComplete(M)) = \depth(M)$. To see that $\HH^n(\dComplete(M)) \to \HH^n(k \ldt_R \dComplete(M))$ is nonzero, first consider the DG-module $\dComplete(M) \ldt_R k$. We have
\begin{align*}
\dComplete(M) \ldt_R k  &\simeq \dComplete(M) \ldt_R \dTorsion(k) \\
                        &\simeq (\dTorsion(\dComplete(M)) \ldt_R k \\
                        &\simeq \dTorsion(M) \ldt_R k \\
                        &\simeq M \ldt_R k
\end{align*}
The first and fourth isomorphisms result from the derived torsion of $k$ and the fact that derived torsion commutes with the derived tensor, as derived torsion is computed by tensoring with a telescope complex. The second isomorphism is from Greenless-May duality. Now consider the commutative diagram

\begin{equation*}
\begin{tikzcd}
M \arrow[r] \arrow[d] & M \ldt_R k \arrow[d, "\simeq"] \\
\dComplete(M) \arrow[r] & \dComplete(M) \ldt_R k
\end{tikzcd}
\end{equation*}
After taking the $n$th cohomology, the hypothesis that $M$ has maximal depth implies that the top row is nonzero. The commutativity and isomorphism force the bottom row to be nonzero. Thus, $\HH^n(\dComplete(M)) \to \HH^n(k \ldt_R \dComplete(M))$ is nonzero.
\end{proof}

The following is a derived analogue of \cite[Lemma 3.5]{MCMC}.

\begin{lemma}
\label{LocalInequality}
Let $R$ be a DG-ring, and $M$ a derived complete DG-module of maximal depth. For $\px \in \Spec \HR$, we have
\begin{equation*}
\depth_{R_\px}(M_\px) \geq \dim(\HR_\px)
\end{equation*}
\end{lemma}
\begin{proof}
Let $\px \in \Spec(\HR)$, and set $h = \text{height } \px$. By \cite[Prop. 3.3]{Koszul}, we have that $\depth_R(\px, M) \leq \depth_{R_\px}(M_\px)$. Let $x_1, \ldots, x_d$ be a system of parameters in $\HR$ such that $x_1\, \ldots, x_h$ is in $\px$. Since $x_1, \ldots, x_d$ is a system of parameters in $\HR$, by \cite[Prop. 3.17]{Koszul} we have
\begin{equation*}
\depth_R(M) = \inf(K(x_1, \ldots, x_d); M) + d
\end{equation*}
which forces $\inf(K(x_1, \ldots, x_d); M) = 0$ as $M$ has maximal depth. By Lemma~\ref{KoszulInf}, we have $\inf(K(x_1, \ldots, x_h); M) \geq 0$. Again by \cite[Prop. 3.17]{Koszul} we have
\begin{align*}
\depth_R(\px, M)    &= \inf(K(x_1, \ldots, x_h); M) + h \\
                    &\geq 0 + h \\
                    &\geq h
\end{align*}
This implies that $h = \dim(\HR_\px)\leq \depth_{R_\px}(M_\px)$.
\end{proof}

We can now prove a version of the Improved New Intersection Theorem for DG-rings of constant amplitude given the existence of a DG-module of maximal depth. This proof adapts \cite[Thm. 3.6]{MCMC}, which establishes the Improved New Intersection Theorem for local rings given the existence of a complex of maximal depth. One important detail to note is that this proof assumes that the DG-ring in question has constant amplitude - otherwise the proof breaks down.

\begin{theorem}
\label{INIT}
Let $R$ be a DG-ring with $\amp(R) = n$ with constant amplitude (i.e., $\Supp(\HH^{\inf(R)}(R)) = \Spec(\HR)$), and $F \in \Dbf(R)$ a semi-free DG-module where $\HH^0(F) \neq 0$ and $\HH^i(F)$ is of finite length for $i \leq -1$. Suppose that there exists a DG-module $M \in \Dbf(R)$ of maximal depth. If an ideal $I$ in $\HR$ annihilates a minimal generator of $\HH^0(F)$, then $\projdim(F)+n \geq \dim \HR - \dim(\HR/I)$. 
\end{theorem}
\begin{proof}
Let $M$ be a DG-module of maximal depth. Since the derived completion of any DG-module of maximal depth is of maximal depth, Lemma~\ref{CompleteMD} allows us to assume that $M$ is derived complete. Let $s = \inf(F \otimes_R M)$. By Lemma~\ref{split}, $s \leq n$.

\sloppy Let $\px \in \Spec(\HR)$ be a minimal prime of $\HH^s(F \otimes_R M)$. This implies that $\depth_{\HR_\px}(\HH^s(F \otimes_R M)_\px) = 0$. Since $(F \otimes_R M)_\px$ is not acyclic, neither $F_\px$ nor $M_\px$ is acyclic. In addition, since $F_\px$ is a bounded, graded-free $R_\px$ DG-module, it is a semi-free $R_\px$ DG-module. Thus, the derived Auslander-Buchsbaum formula implies that
\begin{align}
\projdim_{R_\px} F_\px    &= \depth_{R_\px} (M_\px) - \depth_{R_\px}(F \otimes_R M)_\px \nonumber \\
                            &= \depth_{R_\px}(M_\px) - s \nonumber \\
                            &\geq \dim \HR_\px - s \label{INIT1}
\end{align}
Note that $\depth_{R_\px}(F \otimes_R M)_\px = s$ by Lemma~\ref{DepthBound} since $\px R_\px$ is an associated prime of $\HH^s(F \otimes_R M)_\px$ , and that the above inequality is a consequence of Lemma~\ref{LocalInequality}.

If $s < n$, we shall show that $\px = \mx$. To see that $\px = \mx$, assume otherwise. Since the length of $\HH^i(F)$ is finite for $i \leq -1$, $\HH^i(F_\px) \cong \HH^i(F)_\px = 0$ for $i \leq -1$. This implies that $\inf(F_\px) = 0$, and so $\depth_{R_\px}(F_\px) \geq 0$. This yields
\begin{align*}
\depth_{R_\px}(R_\px)   &= \depth_{R_\px}(F_\px) + \projdim_{R_\px}(F_\px)\\
                &\geq \projdim_{R_\px}(F_\px) \\
                &\geq \dim \HR_\px - s \\
                &> \dim \HR_\px - n                
\end{align*}
Since $\text{Supp}(\HH^{\inf(R)}(R)) = \Spec(\HR)$, we have that $\inf(R_\px) = \inf(R) = -n$. (This step is where the constant amplitude assumption is needed - it is not clear to us if it can be removed in general.) By \cite[Prop. 3.5]{CMDGA}, we have that
\begin{align*}
\depth_{R_\px}(R_\px)   &\leq \dim \HR_\px + \inf(R_\px) \\
                        &\leq \dim \HR_\px - n
\end{align*}
This implies a contradiction, and so $\px = \mx$. As a result, Equation \ref{INIT1} becomes
\begin{align*}
\projdim_R F    &\geq \dim \HR - s \\
                &\geq \dim \HR - n
\end{align*}
which is the desired inequality.

Let $s = n$. Since $\HH^0(F)$ is finitely generated, every minimal generator in $\HH^0(F)$ gives rise to a minimal generator in $\HH^0(F) \otimes_{\HR} k$ by Nakayama's Lemma. By \cite[Lemma 3.2]{FinDim}, $\HH^0(F) \otimes_{\HR} k \simeq \HH^0(F \ldt_R k)$ as $\sup(F) = 0$. The minimal generator in $\HH^0(F \ldt_R k)$ gives rise to a nonzero element in $\HH^n(F \ldt_R (k \otimes_R M))$ by Diagram \ref{splitDiag} in Theorem \ref{split}. The commutativity of Diagram \ref{splitDiag} implies that there is a nonzero element in $\HH^n(M \otimes_R F)$ that is annihilated by $I$, and so $\Gamma_I\HH^n(M \otimes_R F) \neq 0$. Since $\inf(M \otimes_R F) = n$, this implies that $\depth_R(I, M \otimes_R F) = n$.

Now we show that 
\begin{equation*}
\depth_R(M \otimes_R F) \leq \depth_R(I, M \otimes_R F) + \dim(\HR/I)
\end{equation*} Let $\mathbf{a} = a_1, \ldots, a_i$ be a set of generators for $I$, and let $\mathbf{b} = b_1, \ldots, b_j$ be a set of elements whose images in $\HR/I$ form a system of parameters in $\HR/I$. The DG-modules $M \otimes_R F$ and $K(\mathbf{a}; M)$ are derived $\mx$-complete since $M$ is derived $\mx$-complete.  By Lemma ~\ref{KoszulInf}, we have
\begin{equation*}
\inf(K(\mathbf{a}, \mathbf{b}; (M \otimes_R F))) \leq \inf(K(\mathbf{a}; M \otimes_R F))
\end{equation*}
By \cite[Prop. 3.17]{Koszul}, we have $\inf(K(\mathbf{a}; M \otimes_R F)) = \depth_R(I, M \otimes_R F) - i$. The sequence $\mathbf{a}, \mathbf{b}$ generates $\mx$ up to radical, and so $\inf(K(\mathbf{a}, \mathbf{b}; (M \otimes_R F))) = \depth_R(M) - (i + j)$. Thus, the above inequality gives
\begin{align*}
\depth_R(M \otimes_R F) - (i + j) &\leq  \depth_R(I, M \otimes_R F) - i \\
\depth_R(M \otimes_R F) &\leq \depth_R(I, M \otimes_R F) + j \\
            &\leq \depth_R(I, M \otimes_R F) + \dim(\HR/I)
\end{align*}
Combining the last two paragraphs, we have that
\begin{equation*}
\depth(M \otimes_R F) \leq n + \dim(\HR/I) 
\end{equation*}
Using the derived Auslander-Buchsbaum formula (Theorem \ref{ABF}), we get
\begin{align*}
\projdim_R(F)   &= \depth_R(M) - \depth_R(M \otimes_R F) \\
                &\geq \depth_R(M) - (n + \dim(\HR/I))\\
                &\geq \dim(\HR) - \dim(\HR/I) - n
\end{align*}
Thus, we have
\begin{equation*}
\projdim_R(F) + n \geq \dim(\HR) - \dim(\HR/I)
\end{equation*}
in the case $s = n$.
\end{proof}

\section{Maximal Cohen-Macaulay DG-Complexes}
\subsection{Definition}

Shaul \cite{CMDGA} defines a Cohen-Macaulay DG-ring as follows.

\begin{defn}
Let $R$ be a DG-ring. Then $R$ is a \emph{local-Cohen-Macaulay DG-ring} (usually abbreviated to Cohen-Macaulay DG-ring) if $\it\amp(R) = \amp(\dTorsion(R))$. If $R_{\px}$ is local-Cohen-Macaulay for all $\px \in \Spec(\HR)$, then $R$ is \emph{global-Cohen-Macaulay}.
\end{defn}
If a DG-ring $R$ has a dualizing DG-module $D$, then Shaul proves that $R$ is Cohen-Macaulay if and only if $\amp(R) = \amp(D)$.

Expanding the definition of a Cohen-Macaulay DG-ring, Shaul gives the following definition of a Cohen-Macaulay DG-module.

\begin{defn}
Let $R$ be a DG-ring, and let $M$ be a finite DG $R$-module. Then $M$ is a \emph{Cohen-Macaulay} DG-module if $\amp(M) = \amp(\dTorsion(M)) = \amp(R)$. If in addition, $\sup(\dTorsion(M)) = \sup(M) + \dim(\HR)$, then $M$ is a \emph{maximal Cohen-Macaulay DG-module}.
\end{defn}

Shaul proves a number of facts about Cohen-Macaulay DG-rings and Cohen-Macaulay DG-modules in \cite{CMDGA}, most of which generalize facts about Cohen-Macaulay modules over local rings. Some important ones are listed below

\begin{itemize}
	\item If $\dim(\HR) = 0$, then $R$ is a Cohen-Macaulay DG-ring.
	\item $R$ is a Cohen Macaulay DG-ring if and only if there exists $M \in \Dbf(R)$ such that $\amp(M) = \amp(R)$ and $\injdim_R(M) < \infty$.
    \item Let $R$ have a dualizing DG-module $D$. If $M \in \Dbf(R)$ with $\amp(M) = \amp(R)$, then $M$ is Cohen-Macaulay if and only if $\amp(\RHom(M, D)) = \amp(R)$.
    \item If $R$ is a Cohen-Macaulay DG-ring with a dualizing DG-module $D$, then $D$ is a maximal Cohen-Macaulay DG-module.
\end{itemize}

As noted by Shaul in section 6 of \cite{CMDGA}, the definitions of Cohen-Macaulay DG-rings and Cohen-Macaulay DG-modules are generalizations of the classical concepts of Cohen-Macaulay rings and Cohen-Macaulay modules. However, this definition does not extend the notion of a Cohen-Macaulay \emph{complex} over a noetherian local ring (i.e., a complex $M$ over a noetherian local ring $R$ such that $\amp(\dTorsion(M)) = 0$). In particular, the amplitude of a Cohen-Macaulay DG-module must be the same as the amplitude of the DG-ring.

In \cite{MCMC}, Iyengar, Ma, Schwede, and Walker define a maximal Cohen-Macaulay complex over a noetherian local ring, and prove that such complexes exist if the ring has a dualizing complex. In the following paragraph, we propose a definition of a maximal Cohen-Macaulay DG-complex over a noetherian local DG-ring that generalizes the definition of a maximal Cohen-Macaulay complex given in \cite{MCMC}.

\begin{defn}[Maximal Cohen-Macaulay DG-Complex]
\label{MCM}
Let $R$ be a DG-ring with $\amp(R) = n$. A DG R-module $M$ is a maximal Cohen-Macaulay DG-complex if the following conditions hold:
\begin{enumerate}
	\item $M \in \Dbf(R)$
    \item The map $\HH^n(M) \to \HH^n(M \ldt_R k)$ (induced from the morphism of DG-modules $M \to M \ldt_R k$) is nonzero
    \item $\inf(M) = 0$
	\item $\amp(\dTorsion(M)) = n$
    \item $\depth(M) = \dim(\HR)$
\end{enumerate}
\end{defn}
In the case where $\amp(R) = 0$ (i.e., $R$ is a noetherian local ring), the above definition coincides exactly with the definition of a maximal Cohen-Macaulay complex in \cite{MCMC} - noting in this case that condition 3 is extraneous by \cite[Lemma 4.1]{MCMC}.

It is important to note that any maximal Cohen-Macaulay DG-complex is a DG-module of maximal depth, as conditions 2 and 5 of Definition \ref{MCM} immediately imply.

When a DG-ring $R$ has a dualizing DG-module, the following definition of a maximal Cohen-Macaulay DG-complex is equivalent to Definition~\ref{MCM}. This equivalence follows from the DG version of local duality \cite[Thm. 7.26]{Injective}.

\begin{defn}
\label{MCMD}
Let $R$ be a DG-ring with $\amp(R) = n$ and $\dim(\HR) = d$. Suppose $R$ has a dualizing DG-module $D$ such that $D$ is right-normalized $(\inf(D) = 0)$. A DG-module $M$ is a maximal Cohen-Macaulay DG-complex if the following conditions hold:
\begin{enumerate}
	\item $M \in \Dbf(R)$    
    \item The map $\xi_{M^\dag}^{d-n}: \Ext_R^{d-n}(k, M^\dag) \to \Hc^{d-n}(M^\dag)$ (induced from the DG-morphism $R \to k$) is nonzero
    \item $\sup(\dTorsion(M^\dag)) = d$
	\item $\amp(M^\dag) = n$
    \item $\sup(M^\dag) = 0$
\end{enumerate}
\end{defn}
In the case where $\amp(R) = 0$ (i.e., $R$ is a noetherian local ring), the above definition again coincides exactly with \cite{MCMC} - noting that condition 3 is implied by condition 2. 
\begin{theorem}
	Let $R$ be a DG-ring with $\amp(R) = n$ and $\dim(\HR) = d$. Suppose that $R$ has a dualizing DG-module $D$ (which is right-normalized). Then Definitions \ref{MCM} and \ref{MCMD} are equivalent.
\end{theorem}
\begin{proof}
Let $M \in \Dbf(R)$, and suppose that $M$ is a maximal Cohen-Macaulay DG-complex as defined in Definition \ref{MCM}. For each step of the proof, we will need the DG-version of local duality, proved in Theorem 7.26 in \cite{Injective}. we will restate it here for convenience.	
\begin{lemma}[Generalized Local Duality]
Let $M \in \Dbf(R)$. Then
\begin{equation*}
	\dTorsion(M) \simeq \RHom_R(\RHom_R(M, \Sigma^{d} D), E)
\end{equation*}	
where $E$ is the generalized injective hull of $R$. By passing to cohomology we also have
\begin{equation*}
	\Hc^i(M) \simeq \Hom_{\HR}(\Ext_R^{d-i}(M, D), \HH^0(E))
\end{equation*}	
for every $i \in \mathbb{Z}$. Note that $\HH^0(E)$ is the injective hull of the local ring $\HR$.
\end{lemma}	
To show condition 3, note that by local duality, we have	
\begin{equation*}
	\dTorsion(M^\dagger) \simeq \RHom_R(\Sigma^{d}(M^{\dagger \dagger}), E)
\end{equation*}	
which simplifies to
\begin{equation*}
	\dTorsion(M^\dagger) \simeq \Sigma^{-d}\RHom_R(M, E)
\end{equation*}	
Since $\inf(M) = 0$, we have $\sup(\RHom_R(M, E)) = 0$ by \cite[Thm. 4.10]{Injective}. This implies that $\sup(\Sigma^{-d}\RHom_R(M, E)) = d$, which gives $\sup(\dTorsion(M^{\dagger})) = d$.
	
To show condition 4, we have that $\amp(\dTorsion(M)) = n$ by assumption, and so \begin{equation*}
\amp(\RHom_R(M^{\dagger}, E)) = n
\end{equation*} by local duality. This implies that $\amp(M^{\dagger}) = n$ by \cite[Thm. 4.10]{Injective}.
	
To show condition 5, we have that $\depth(M) = d$ by assumption. This forces $\inf(\dTorsion(M)) = d$ by the definition of depth, and so $\inf(\RHom_R(M^{\dagger}, E)) = 0$ by local duality. By Theorem 4.10 in \cite{Injective}, this forces $\sup(M^{\dagger}) = 0$.
	
Finally, to show condition 2, consider the map $\HH^n(M) \to \HH^n(k \ldt_R M)$, which is nonzero by assumption. Since the biduality morphism is an isomorphism in the derived category, we have that the map $\HH^n(M^{\dag \dag}) \to \HH^n((k \ldt_R M)^{\dag\dag})$ is nonzero. Applying the $\Hom_{\HR}(-, \HH^0(E))$ functor yields the map
\begin{equation} \label{eq:1}
	\Hom_{\HR}(\HH^n((k \ldt_R M)^{\dag\dag}), \HH^0(E)) \to \Hom_{\HR}(\HH^n(M^{\dag\dag}), \HH^0(E))
\end{equation}
which is nonzero since $\HH^0(E)$ is a faithfully injective $\HR$-module. By \cite[Thm. 4.10]{Injective}, Equation \ref{eq:1} can be rewritten as
\begin{equation} \label{eq:2}
	\HH^{-n}(\RHom_R((k \ldt_R M)^{\dag\dag}, E)) \to \HH^{-n}(\RHom_R((M)^{\dag\dag}, E))
\end{equation}	
The object $\RHom_R((M)^{\dag\dag}, E)$ is equivalent to $\Sigma^d(\dTorsion(M^\dag))$ by local duality. We can rewrite $\RHom_R((k \otimes_R M)^{\dag\dag}, E)$ through the following series of transformations:	
\begin{align*}
\RHom_R((k \ldt_R M)^{\dag\dag}, E) 	
	&\simeq \RHom_R(k \ldt_R M, E) \\		
	&\simeq \RHom_R(k, \RHom_R(M, E))\\
	&\simeq \RHom_R(k, \RHom_R(M, \Sigma^d(\dTorsion(D)))\\
	&\simeq \Sigma^d \RHom_R(k, \RHom(M, \dTorsion(D))\\
	&\simeq \Sigma^d \RHom_R(\dTorsion(k), \RHom_R(M, D))\\
	&\simeq \Sigma^d \RHom_R(k, M^{\dag})
\end{align*}							 			
The third isomorphism holds since $E \simeq \Sigma^d\dTorsion(D)$ by \cite[Prop. 7.25]{Injective}, and the fifth isomorphism holds by \cite[Prop. 7.24]{Injective} and Greenless-May duality. The sixth isomorphism holds since the DG-module $k$ is $\mx$-torsion. Equation \ref{eq:2} now becomes	
\begin{equation} \label{eq:3}
	\HH^{-n}(\Sigma^d \RHom_R(k, M^{\dag})) \to \HH^{-n}(\Sigma^d(\dTorsion(M^{\dag})))
\end{equation}
which simplifies to
\begin{equation*}
	\Ext^{d-n}_R(k, M^{\dag}) \to \Hc^{d-n}(M^{\dag})
\end{equation*}	
Thus, condition 2 is satisfied, and $M$ is a maximal Cohen-Macaulay DG-complex as given in Definition \ref{MCMD}.
Each step of the forward direction can be straightforwardly reversed to obtain the converse.
\end{proof}

\subsection{Maximal Cohen-Macaulay DG-Modules and DG-Complexes}

When restricting to the case of DG-rings, our definition of a maximal Cohen-Macaulay DG-complex and Shaul's definition of a maximal Cohen-Macaulay DG-module coincide.

\begin{theorem}
Let $R$ be a DG-ring with $\amp(R) = n$. Then $R$ is a local Cohen-Macaulay DG-ring (as defined by Shaul) if and only if the DG $R$-module $\Sigma^{-n}R$ is a maximal Cohen-Macaulay DG-complex (as defined in Definition~\ref{MCM}).
\end{theorem}
\begin{proof}
First suppose that $R$ is a local Cohen-Macaulay DG-ring. We immediately have that $R \in \Dbf(R)$ and so $\Sigma^{-n}R \in \Dbf(R)$. Since $R$ is local Cohen-Macaulay, $\amp(\dTorsion(\Sigma^{-n}R)) = \amp(\dTorsion(R)) = n$ by definition. The map $\HH^n(\Sigma^{-n}R) \to \HH^n(\Sigma^{-n}R \ldt_R k)$ can be rewritten as $\HH^0(R) \to \HH^0(R \ldt_R k)$, which is just the canonical surjection $\HR \to k$. Since $\inf(R) = -n$, $\inf(\Sigma^{-n}R) = 0$. Finally, the sequential depth of $\Sigma^{-n}R$ is $\dim(\HR)$ by Proposition 6.10 in \cite{CMDGA}, and since $\inf(\Sigma^{-n}R)=0$, we have that $\depth(\Sigma^{-n}R) = d$.

Conversely, if $R$ is not a local Cohen-Macaulay DG-ring, we have
\begin{equation*}
\amp(\dTorsion(\Sigma^{-n}R)) = \amp(\dTorsion(R)) \neq \amp(R)
\end{equation*}
which implies that $\Sigma^{-n}R$ cannot be a maximal Cohen-Macaulay DG-complex.
\end{proof}

However, there is a difference between the two definitions in the case of DG-modules. Every maximal Cohen-Macaulay DG-module (as defined by Shaul) is a Cohen-Macaulay DG-complex (up to a shift). However, maximal Cohen-Macaulay DG-modules must have the same amplitude as the base DG ring. But even in the ring case, there exist maximal Cohen-Macaulay complexes that are not modules, and so cannot be maximal Cohen-Macaulay DG-modules according to Shaul's definition. Thus, not every maximal Cohen-Macaulay DG-complex is a maximal Cohen-Macaulay DG-module (as defined by Shaul).

\section{The Canonical Element Theorem for DG-Rings}

Let $R$ be a DG-ring with right normalized dualizing DG-module $D$. By the definition of a DG-ring, $\HR$ is a noetherian local ring with maximal ideal $\mx$ and residue field $k$. By \cite[Prop. 7.5]{Tilting}, $\RHom_R(\HR, D)$ is a dualizing complex for $\HR$, and since $\inf(\RHom_R(\HR, D)) = \inf(D)$ by Proposition 3.3 in \cite{Injective}, $\RHom_R(\HR, D)$ is a right-normalized dualizing complex for $\HR$. Again by  \cite[Prop. 3.3]{Injective}, we have $\HD \cong \HH^0(\RHom_R(\HR, D)) $. By \cite[Thm 4.5, Lemma 4.8]{MCMC}, the map $\Ext^d_{\HR}(k, \HD) \to \HH_{\mx\HR}^d(\HD)$ is nonzero (where $d$ is the dimension of $\HR$). This result is equivalent to the Canonical Element Theorem for noetherian local rings. We will use this map to conclude that the map $\Ext^d_R(k, \HD) \to \Hc[\mx R]^d(\HD)$ is nonzero. This will establish a Canonical Element Theorem for DG-rings.

Before we prove the main result, we will establish independence of base change for derived torsion over DG-rings. This is similar to Proposition 2.10 in \cite{CMDGA}, but establishes independence of base change through restriction of scalars, not through extension of scalars.

\begin{lemma}
\label{IndBase}
	Let $R$ and $S$ be DG-rings equipped with a DG-ring map $f: R\to S$. Consider any ideal $\mathfrak{a} \subset \HR$, and let $\mathfrak{b} \subset \HS$ be the extension of $\mathfrak{a}$ to $\HS$ along $f$. If $M$ is a DG $S$-module, then $\dTorsion[\mathfrak{a}](M) \simeq \dTorsion[\mathfrak{b}](M)$ (viewing $M$ as a DG $R$-module through restriction of scalars). 
\end{lemma}
\begin{proof}
Let $\overline{r}_1, \ldots, \overline{r}_n$ be a sequence of elements in $\HR$ that generates $\mathfrak{a}$, and let $r_1, \ldots, r_n$ be lifts of these elements to $R^0$. By Corollary 2.13 in \cite{DerCom}, we have that
\begin{equation*}
	\dTorsion[\mathfrak{b}](M) \simeq \tel(S^0; f(r_1), \ldots, f(r_n)) \otimes_{S^0} S \otimes_S M
\end{equation*}
By applying the base change property of the telescope complex to the map $R^0 \to S^0$, we have that
\begin{equation*}
	\tel(R^0; r_1, \ldots, r_n) \otimes_{R^0} S^0 \simeq \tel(S^0; f(r_1), \ldots, f(r_n)) 
\end{equation*}
Rewriting the first equation, we have
\begin{align*}
	\dTorsion[\mathfrak{b}](M) 	&\simeq \tel(S^0; f(r_1), \ldots, f(r_n)) \otimes_{S^0} S \otimes_S M \\
						  			&\simeq \tel(R^0; r_1, \ldots, r_n) \otimes_{R^0} S^0 \otimes_{S^0} S \otimes_S M \\
						  			&\simeq \tel(R^0; r_1, \ldots, r_n) \otimes_{R^0} R \otimes_R M \\
						  			&\simeq \dTorsion[\mathfrak{a}] (M)
\end{align*}
where the last line comes from \cite[Cor. 2.13]{DerCom}.
\end{proof}

\begin{theorem}[Canonical Element Theorem for DG-rings]
\label{zeromap}
Let $R$ be a DG-ring with right normalized dualizing DG-module $D$. The map of DG $R$-modules
\begin{equation*}
\xi^d_{\HD}: \Ext^d_R(k, \HD) \to \Hc[\mx R]^d(\HD)
\end{equation*} is nonzero.
\end{theorem}
\begin{proof}
Consider the map $R \to \HR$ of DG $R$-modules. By applying the $-\ldt_R k$ functor we obtain the map $R \ldt_R k \to \HR \ldt_R k$ of DG $\HR$-modules. This reduces to the map $k \to \HR \ldt_R k$ of DG $\HR$-modules. By applying the $\RHom_{\HR}(-, \HD)$ functor we obtain the map 
\begin{equation}
\phi: \RHom_{\HR}(\HR \ldt_R k, \HD) \to \RHom_{\HR}(k, \HD)
\end{equation}
To compute $\phi$, first consider a minimal semi-free resolution $F$ of $\HR$ (with respect to $R$). By the properties of the derived tensor functor, $\HR \ldt_R k$ is given by $F \otimes_R k$. Since $F$ is a minimal semi-free resolution of $\HR$, the $\HR$ complex $F \otimes_R k$ has trivial differential, and $\sup(F \otimes_R k) = 0$ with $\HH (F \otimes_R k) = (F \otimes_R k)^0 = k$. We then have that $\RHom_{\HR}(\HR \ldt_R k, \HD)$ is equivalent to $\RHom_{\HR}(F \otimes_R k, \HD)$. Let $E^\bullet$ be an injective resolution of $\HD$ (with respect to $\HR$). By the properties of the derived Hom functor, $\RHom_{\HR}(F \otimes_R k, \HD)$ is given by $\Hom_{\HR}(F \otimes_R k, E^\bullet)$. The first three nonzero components of this complex are shown below. 

\begin{center}
\begin{tikzcd}[row sep = 0.2 em]
&& \Hom_R(k, E^2) \\
& \Hom_R(k, E^1) \arrow[ur] \arrow[dr] & \bigoplus\\
\Hom_R(k, E^0) \arrow[ur] \arrow[dr] & \bigoplus & \Hom_R((F \otimes_R k)_1, E^1) \\
& \Hom_R((F \otimes_R k)_1, E^0) \arrow[ur]\arrow[dr] & \bigoplus \\
&& \Hom_R((F \otimes_R k)_2, E^0)
\end{tikzcd}
\end{center}

Since $F \otimes_R k$ has trivial differential, each downward diagonal arrow is zero. This forces the $i$th homology of $\RHom_{\HR}(\HR \ldt_R k, \HD)$ to be the direct sum $\bigoplus\limits_{0 \leq j \leq i} H_{i-j}(\Hom(F \otimes_R k)_j, E^\bullet)$. Thus, the cohomology of $\RHom_{\HR}(k, \HD)$ is a direct summand of the cohomology of $\RHom_{\HR}(\HR \ldt_R k, \HD)$, and the induced map on cohomology of $\phi$ is the projection map.

The map $\phi$ fits into the commutative diagram given below

\begin{center}
\begin{tikzcd}
    \RHom_{\HR}(\HR \ldt_R k, \HD) \arrow[r] \arrow[d, "\phi"] & \dTorsion[\mx R](\HD) \arrow[d, "\simeq"] \\ \RHom_{\HR}(k, \HD) \arrow[r] & \dTorsion[\mx\HR](\HD)
\end{tikzcd}
\end{center}

where the rightmost arrow is a isomorphism by the preceding lemma (with respect to the map $R \to \HR$). Since $\RHom_{\HR}(\HR \ldt_R k, \HD) \simeq \RHom_{R}(k, \HD)$ by derived Hom-tensor adjunction, we have the diagram

\begin{center}
\begin{tikzcd}
    \RHom_{R}(k, \HD) \arrow[r] \arrow[d, "\phi"] & \dTorsion[\mx R](\HD) \arrow[d, "\simeq"] \\ \RHom_{\HR}(k, \HD) \arrow[r] & \dTorsion[\mx \HR](\HD)
\end{tikzcd}
\end{center}

which by taking homology induces the following diagram

\begin{center}
\begin{tikzcd}
    \Ext_R^d(k, \HD) \arrow[r] \arrow[d, "\overline{\phi}_d"] & \Hc[\mx R]^d(\HD) \arrow[d, "\cong"] \\
    \Ext_{\HR}^d(k, \HD) \arrow[r] & \Hc[\mx\HR]^d(\HD)
\end{tikzcd}
\end{center}

The Canonical Element Theorem asserts that the bottom map is nonzero. Since $\Ext^d_{\HR}(k, \HD)$ is a direct summand of $\Ext^d_R(k, \HD)$, the top map must be nonzero.
\end{proof}

\begin{theorem}
\label{Existence}
Let $R$ be a complete DG-ring with constant amplitude and $\amp(R) = n$. If $D$ is a right-normalized dualizing DG-module, then the DG-module $\RHom_R(\Sigma^n(D^{\leq n}), D)$ is a maximal Cohen-Macaulay DG-complex.
\end{theorem}

\begin{proof}
We will use Definition~\ref{MCMD} to prove that $\RHom_R(\Sigma^n(D^{\leq n}), D)$ is a maximal Cohen-Macaulay DG-complex. Consider the DG-module $N = \Sigma^n(D^{\leq n})$. By the properties of a dualizing DG-module, $\RHom_R(N, D)^\dag = N$, and so we must verify that $N$ satisfies conditions 1-5 in Definition~\ref{MCMD}. Condition 1 is immediate. We have that $\inf(N) = -n$ and that $\amp(N) \leq n$. To show conditions 4 and 5, we must show that $\HH^0(N) \neq 0$. Note that $D_\px$ is a dualizing DG-module for $R_\px$ for any $\px \in \Spec{\HR}$ by \cite[Cor. 6.11]{Twisted}. For any prime $\px \in \text{Min}(\HR)$, the DG-ring $R_\px$ is Cohen-Macaulay, and so $\amp(D_\px) = \amp(R_\px)$. This implies that $\Supp(\HH^n(D)) = \Spec(\HR)$, and so $\HH^n(D) \neq 0$ and $\dim(\HH^n(D)) = d$. Since $\HH^n(D) \simeq \HH^0(N)$, we have that $\HH^0(N) \neq 0$. Thus, conditions 4 and 5 are satisfied. In addition, since $\dim(\HH^0(N)) = d$, the supremum of $\dTorsion(N)$ is $d$ \cite[Thm 2.15]{CMDGA}, and so condition 3 is satisfied.\\
To prove condition 2, we will argue as in the proof of \cite[Lemma 4.8]{MCMC}. By Theorem \ref{zeromap} that the following map is nonzero.
\begin{equation} \label{eq:4}
\xi_{\HD}^d: \Ext_R^d(k, \HH^0(D)) \to \Hc^d(\HH^0(D))
\end{equation}
Since $\HE$ is a faithfully injective $\HR$-module, there exists a map 
\begin{equation*}
\alpha \in \Hom_{\HR}(\Hc^d(\HD), \HE)
\end{equation*}
such that $\alpha \circ \xi_{\HD}^d \neq 0$. We have that $\HR$ is complete since $R$ is a complete DG-ring \cite[Prop. 1.7]{CMDGA}. Since $\HR$ is complete, we can consider $\alpha$ to be induced by a map \begin{equation*}
f \in \HH^0(\RHom_{\HR}(\HD, \RHom_R(\HR, D)))
\end{equation*}
by local duality. (Keep in mind that the dualizing complex for $\HR$ is $\RHom_R(\HR, D))$.) By derived Hom-tensor adjunction we have that
\begin{equation*}
\HH^0(\RHom_{\HR}(\HD, \RHom_R(\HR, D))) \cong \HH^0(\RHom_R(\HD, D)
\end{equation*}
\sloppy and so $f$ can be considered a map in $\HH^0(\RHom_R(\HD, D))$, or alternately $\Hom_{\sD(R)}(\HD, D)$. Any morphism in $\Hom_{\sD(R)}(\HD, D)$ can be considered to be a morphism in $\Hom_{\sD(R)}(\HD,\\ D^{\leq n})$ since $\sup(\HD) = 0$ and $\inf(D) = 0$. We then have the commutative diagram
\begin{equation*}
\begin{tikzcd}[column sep = large, row sep = large, font = \normalsize]
\Ext_R^d(k, \HH^0(D)) \arrow[r, "\xi_{\HD}^d"]\arrow[d, "{\Ext_R^d(k, f)}"] & \Hc^d(\HH^0(D))\arrow[rd, "\alpha"]\arrow[d, "{\Hc^d(f)}"] & \\
\Ext_R^d(k, D^{\leq n}) \arrow[r, "\xi_{D^{\leq n}}^d"] & \Hc^d(D^{\leq n}) \arrow[r] & \HE
\end{tikzcd}
\end{equation*}
Since $\alpha \circ \xi_{\HD}^d$ is nonzero, the following map is nonzero.
\begin{equation*}
\xi^d_{D^{\leq n}}: \Ext_R^d(k, D^{\leq n}) \to \Hc^d(D^{\leq n})
\end{equation*}
Rewriting $D^{\leq n}$ as $\Sigma^{-n}(N)$, we obtain the nonzero map
\begin{equation*}
\xi^{d-n}_N: \Ext_R^{d-n}(k, N) \to \Hc^d(N)
\end{equation*}
Thus, condition 2 is satisfied.
\end{proof}

As in the ring case, it is possible to extend Theorem \ref{Existence} to all DG-rings with constant amplitude admitting a dualizing DG-module. We will first prove a necessary lemma about the dualizing DG-module of the derived completion of a DG-ring.
\begin{lemma}
\label{DualizingComplete}
Let $R$ be a DG-ring with a right-normalized dualizing DG-module $D$. The DG $\dComplete(R)$-module $D \ldt_R \dComplete(R)$ is a right-normalized dualizing DG-module for $\dComplete(R)$.
\end{lemma}
\begin{proof}
Throughout this proof, we will denote $\dComplete(R)$ as $\Rhat$.

We will first prove that the DG $\Rhat$-module $D \ldt_R \Rhat$ is a right-normalized dualizing DG-module for $\Rhat$. Since $D$ is a dualizing DG $R$-module, the homothety morphism $R \to \RHom_R(D, D)$ is an isomorphism. Tensoring with $\Rhat$ (after potentially replacing $\Rhat$ with a quasi-isomorphic DG-ring to guarantee a map from $R$ to $\Rhat$) gives the isomorphism $\Rhat \to \RHom_R(D, D)\\ \ldt_R \Rhat$ of DG $\Rhat$-modules. Since $\Rhat$ has finite flat dimension with respect to $R$ \cite[Cor. 4.6]{HomDim}, derived tensor evaluation can be used in this situation. We have the following isomorphisms
\begin{align*}
	\Rhat 	&\simeq \RHom_R(D, D) \ldt_R \Rhat \\
	&\simeq \RHom_R(D, D \ldt_R \Rhat) \\
	&\simeq \RHom_R(D, \RHom_{\Rhat}(\Rhat, D\ldt_R \Rhat)) \\
	&\simeq \RHom_{\Rhat}(D \ldt_R \Rhat, D \ldt_R \Rhat)
\end{align*}
where the second isomorphism is from derived tensor evaluation \cite[Prop. 14.3.19]{DC}, and the fourth is derived Hom-tensor evaluation. Thus, the homothety morphism $\Rhat \to \RHom_{\Rhat}(D \ldt_R \Rhat, D \ldt_R \Rhat)$ is an isomorphism.

We will now prove that $\injdim_{\Rhat}(D \ldt_R \Rhat) < \infty$. Since $D$ is a dualizing DG-module for $R$ by assumption, we have that $\injdim_R(D) < \infty$. By \cite[Prop. 2.5]{Injective}, this implies that \\$\injdim_{\HR}(\RHom_R(\HR, D)) < \infty$. Since $\HR$ is a noetherian local ring, and $\widehat{\HR}$ is a faithfully flat $\HR$-module, we have 
\begin{equation*}
\injdim_{\HR}(\RHom_R(\HR, D) \otimes_{\HR} \widehat{\HR}) = \injdim_{\HR}(\RHom_R(\HR, D))
\end{equation*}
by \cite[Thm 15.4.35]{DCMCA}. We have $\HR \ldt_R \Rhat \simeq \widehat{\HR}$ by \cite[Lemma 4.6]{Min2}, and so we have the following set of isomorphisms in $\sD(\HR)$
\begin{align*}
\RHom_R(\HR, D) \otimes_{\HR} \widehat{\HR} 
	&\simeq \RHom_R(\HR, D) \otimes_{\HR} (\HR \ldt_R \Rhat) \\
	&\simeq \RHom_R(\HR, D) \ldt_R \Rhat \\
	&\simeq \RHom_R(\HR, D \ldt_R \Rhat)
\end{align*}
by \cite[Thm 14.3.19]{DC}. By \cite[Thm 17.3.16]{DCMCA}, we have
\begin{equation*}
\injdim_{\HR}(\RHom_R(\HR, D \ldt_R \Rhat)) =\\ \injdim_{\widehat{\HR}}(\RHom_{\HR}(\widehat{\HR}, \RHom_R(\HR, D \ldt_R \Rhat)))
\end{equation*}
We then have the final set of equalities
{\allowdisplaybreaks
\begin{align*}
\injdim_{\HR}(\RHom_R(\HR, D)) 	&=\injdim_{\HR}(\RHom_R(\HR, D \ldt_R \Rhat)) \\
								&=\injdim_{\widehat{\HR}}(\RHom_{\HR}(\widehat{\HR}, \RHom_R(\HR, D \ldt_R \Rhat)))\\
								&=\injdim_{\widehat{\HR}}(\RHom_{R}(\widehat{\HR}, D \ldt_R \Rhat))\\
								&=\injdim_{\widehat{\HR}}(\RHom_{\Rhat}(\widehat{\HR}, D \ldt_R \Rhat))
\end{align*}
}which holds through derived Hom-tensor adjunction. The injective dimension of\\ $\RHom_{\Rhat}(\widehat{\HR}, D \ldt_R \Rhat)$ must then be finite, and so $\injdim_{\Rhat}(D \ldt_R \Rhat) < \infty$ by \cite[Prop. 2.5]{Injective}. (Keep in mind that $\widehat{\HR} \simeq \HH^0(\Rhat)$ by \cite[Prop. 1.7]{CMDGA}.)

Thus, the DG $\Rhat$-module $D \ldt_R \Rhat$ is dualizing, and since $\inf(D \ldt_R \Rhat) = \inf(D)$ by \cite[Lemma. 4.8]{FinDim}, $D \ldt_R \Rhat$ is a right-normalized dualizing DG $\Rhat$-module. 	
\end{proof}

We can now prove the existence of maximal Cohen-Macaulay DG-complexes over DG-rings of constant amplitude admitting a dualizing DG-module.

\begin{theorem}[Existence of Maximal Cohen-Macaulay DG-Complexes]
\label{ExistenceGeneral}
Let $R$ be a DG-ring with constant amplitude and $\amp(R) = n$. If $D$ is a right-normalized dualizing DG-module, then the DG-module $\RHom_R(\Sigma^n(D^{\leq n}), D)$ is a maximal Cohen-Macaulay DG-complex.
\end{theorem}
\begin{proof}
As in the previous proof, we will use Definition \ref{MCMD} to prove that $\RHom_R(\Sigma^n(D^{\leq n}), D)$ is a maximal Cohen-Macaulay DG-complex. By the same reasoning given in Theorem \ref{Existence}, conditions 1, 3, 4, and 5 are satisfied. The only non-trivial condition is showing that the map
\begin{equation*}
\xi^d_{D^{\leq n}}: \Ext^d_R(k, D^{\leq n}) \to \Hc^d(D^{\leq n})
\end{equation*}
is nonzero. We will prove this by passing to the DG-ring $\dComplete(R)$. For the rest of the proof, we will denote $\dComplete(R)$ as $\Rhat$.\\

By \cite[Cor. 5.14]{DerCom}, \cite[Lemma 1.7]{CMDGA}, and \cite[Lemma 4.1]{Koszul}, $\Rhat$ is a complete  DG-ring with constant amplitude such that $\amp(\Rhat) = \amp(R)$. In addition, we have that $D \ldt_R \Rhat$ is a right-normalized dualizing DG-module for $\Rhat$ by Lemma \ref{DualizingComplete}. Thus, by Theorem \ref{Existence}, we have that the map
\begin{equation*}\label{eq:5}
\Ext_{\Rhat}^d(k, (D \ldt_R \Rhat)^{\leq n}) \to \Hc[\widehat{\mx}]^d((D \ldt_R \Rhat)^{\leq n})
\end{equation*}
is nonzero.

We now claim that $(D \ldt_R \Rhat)^{\leq n} \simeq D^{\leq n} \ldt_R \Rhat$. To show this, consider the distinguished triangle
\begin{equation*}
\begin{tikzcd}[column sep = small]
D^{\leq n} \arrow[r] & D \arrow[r] & D^{> n} \arrow[r] & \phantom{}
\end{tikzcd}
\end{equation*}
Applying first the triangulated functor $- \ldt_R \Rhat$ and then the triangulated functor $(-)^{\leq n}$ yields the distinguished triangle
\begin{equation*}
\begin{tikzcd}[column sep = small]
(D^{\leq n} \ldt_R \Rhat)^{\leq n} \arrow[r] & (D \ldt_R \Rhat)^{\leq n} \arrow[r] & (D^{> n} \ldt_R \Rhat)^{\leq n} \arrow[r] & \phantom{}
\end{tikzcd}
\end{equation*}
By \cite[Lemma 4.6]{Min2}, we have that $(D^{> n} \ldt_R \Rhat)^{\leq n} = 0$, and so $(D^{\leq n} \ldt_R \Rhat)^{\leq n} \simeq (D \ldt_R \Rhat)^{\leq n}$.
Again by Lemma 4.6, we have that $\HH^i((D^{\leq n} \ldt_R \Rhat)^{\leq n}) = 0$ for all $i > n$, and so $(D^{\leq n} \ldt_R \Rhat)^{\leq n} \simeq D^{\leq n} \ldt_R \Rhat$. This establishes that 
\begin{equation*}
(D \ldt_R \Rhat)^{\leq n} \simeq D^{\leq n} \ldt_R \Rhat
\end{equation*}

We now have that Equation \ref{eq:5} can be written as
\begin{equation*}
	\Ext_{\Rhat}^d(k, D^{\leq n} \ldt_R \Rhat) \to \Hc[\widehat{\mx}]^d(D^{\leq n} \ldt_R \Rhat)
\end{equation*}
Using the definitions of Ext and local cohomology, we can view the above map as
\begin{equation*}
	\HH^d(\RHom_{\Rhat}(k, D^{\leq n} \ldt_R \Rhat)) \to \HH^d(\dTorsion[\widehat{\mx}](D^{\leq n} \ldt_R \Rhat))
\end{equation*}
We have that 
\begin{align*}
\RHom_{\Rhat}(k, D^{\leq n} \ldt_R \Rhat) 	
	&\simeq \RHom_R(k, D^{\leq n} \ldt_R \Rhat) \\
	&\simeq \RHom_R(k, D^{\leq n}) \ldt_R \Rhat
\end{align*}
through derived Hom-tensor adjunction and derived tensor evaluation. Similarly, we have 
\begin{align*}
\dTorsion[\widehat{\mx}](D^{\leq n} \ldt_R \Rhat)
	&\simeq \dTorsion(D^{\leq n} \ldt_R \Rhat) \\
	&\simeq \dTorsion(D^{\leq n}) \ldt_R \Rhat	  									
\end{align*}
by Lemma \ref{IndBase} and because derived torsion commutes with the derived tensor. Thus, the map
\begin{equation*}
	\HH^d(\RHom_{\Rhat}(k, D^{\leq n} \ldt_R \Rhat)) \to \HH^d(\dTorsion[\widehat{\mx}](D^{\leq n} \ldt_R \Rhat))
\end{equation*}
is equivalent to
\begin{equation*}
	\HH^d(\RHom_R(k, D^{\leq n}) \ldt_R \Rhat) \to \HH^d(\dTorsion(D^{\leq n}) \ldt_R \Rhat)
\end{equation*}
By \cite[Lemma 4.6]{Min2}, the above map equals
\begin{equation*}
	\HH^d(\RHom_R(k, D^{\leq n})) \otimes_{\HR} \widehat{\HR} \to \HH^d(\dTorsion(D^{\leq n})) \otimes_{\HR} \widehat{\HR}
\end{equation*}
Since this map is nonzero, we have that the map
\begin{equation*}
	\HH^d(\RHom_R(k, D^{\leq n})) \to \HH^d(\dTorsion(D^{\leq n}))
\end{equation*}
is nonzero by the faithful flatness of $\widehat{\HR}$. Rewriting this map with Ext and local cohomology, we have that
\begin{equation*}
	\Ext^d_R(k, D^{\leq n}) \to \Hc^d(D^{\leq n})
\end{equation*}
is nonzero, which completes the proof.
\end{proof}

As a result of Theorem \ref{ExistenceGeneral}, all DG-rings with constant amplitude admitting a dualizing DG-module have a maximal Cohen-Macaulay DG-complex. Since any maximal Cohen-Macaulay DG-complex is a DG-module of maximal depth, such DG-rings satisfy the Derived Improved New Intersection Theorem. However, the validity of the Derived Improved New Intersection Theorem can be easily extended to all DG-rings with constant amplitude by passing to the completion.

\begin{theorem}[The Derived Improved New Intersection Theorem] \label{INITGeneral}
Let $R$ be a DG-ring with $\amp(R) = n$ with constant amplitude (i.e., $\Supp(\HH^{\inf(R)}(R)) = \Spec(\HR)$), and $F \in \Dbf(R)$ a semi-free DG-module where $\HH^0(F) \neq 0$ and $\HH^i(F)$ is of finite length for $i \leq -1$. If an ideal $I$ in $\HR$ annihilates a minimal generator of $\HH^0(F)$, then $\projdim(F)+n \geq \dim \HR - \dim(\HR/I)$.
\end{theorem}

\begin{proof}
The DG-ring $\dComplete(R)$ has constant amplitude, has $\amp(\Rhat) = n$, and admits a dualizing DG-module \cite[Lemma 4.1]{Koszul} \cite[Prop. 7.21]{Injective}. Thus, by Theorem \ref{Existence}, $\dComplete(R)$ has a maximal Cohen-Macaulay DG-complex. In addition, the DG $\Rhat$-module $F \ldt_R \Rhat$ is a semi-free DG $\Rhat$-module, with $\HH^i(F \ldt_R \Rhat) = \widehat{\HH^i(F)}$. The ideal $\widehat{I}$ in $\widehat{\HR}$ annihilates a minimal generator of $\HH^0(F \ldt_R \Rhat)$ by the exactness of completion. Thus, by Theorem \ref{INIT}, we have that
\begin{equation*}
	\projdim_{\Rhat}(F \ldt_R \Rhat)+n \geq \dim \widehat{\HR} - \dim(\widehat{\HR/I})
\end{equation*}
But $\dim(\widehat{\HR}) = \dim(\HR)$, $\dim(\widehat{\HR/I}) = \dim(\HR/I)$, and $\projdim_{\Rhat}(F \ldt_R \Rhat) = \projdim_R(F)$ (the last fact is due to \cite[Lemma 4.9]{FinDim}). Thus, we have
\begin{equation*}
	\projdim_{R}(F)+n \geq \dim \HR - \dim(\HR/I)
\end{equation*}
which proves the theorem.
\end{proof}

Thus, the Derived Improved New Intersection Theorem holds for all DG-rings of constant amplitude, regardless of whether they admit a dualizing DG-module.

It is natural to wonder if it is possible to lift the assumption that the DG-rings in question have constant amplitude. Since the proofs of the Derived Improved New Intersection Theorem and the existence of maximal Cohen-Macaulay DG-complexes both rely on localization, it does not seem likely to the author that the assumption can be dropped.

\subsection*{Acknowledgments}
The author wishes to thank his thesis advisor, Tom Marley, for both suggesting the topic behind this paper, and for the many discussions and insights that shaped this paper. He also wishes to thank the anonymous referee for comments and suggestions that improved this paper.

\bibliographystyle{plain}
\bibliography{MCM_DG-Complexes}

\end{document}